\documentclass{amsart}
\usepackage{newtxtext}
\usepackage{mathrsfs}
\usepackage[frenchmath]{newtxmath}

\usepackage{mathtools}		
\usepackage{geometry}
\usepackage{bm}
\usepackage{microtype}
\usepackage{hyperref}
\hypersetup{colorlinks=true, citecolor=blue, linkcolor=blue}

\newcommand{\Nbb}{\mathbb N}

\newcommand{\Rbb}{\mathbb R}

\newcommand{\Zbb}{\mathbb Z}

\newcommand{\Fcal}{\mathcal F}

\newcommand{\Kcal}{\mathcal K}
\newcommand{\Lcal}{\mathcal L}
\newcommand{\Mcal}{\mathcal M}
\newcommand{\Ncal}{\mathcal N}

\newcommand{\Pcal}{\mathcal P}
\newcommand{\Qcal}{\mathcal Q}

\newcommand{\Tcal}{\mathcal T}

\newcommand{\Cscr}{\mathscr C}

\newcommand{\Pscr}{\mathscr P}

\renewcommand{\to}{\rightarrow}

\newcommand{\dd}{\:\mathrm{d}}

\DeclareMathOperator{\Sdiff}{\triangle}

\def\mmath#1{\text{\scalebox{1.1}{$#1$}}}
\def\mfrac#1#2{\mmath{\frac{#1}{#2}}}
\def\sub#1{_\mmath{#1}}
\newcommand{\Linfty}{\Lcal_\infty(G)}
\newcommand{\LinftyS}{\Lcal_\infty^*(G)}
\newcommand{\Lone}{\Lcal_1(G)}
\newcommand{\Cp}{\Cscr^+}

\newcommand{\cf}{\textit{cf.}}

\newcommand{\ie}{\textit{i.e.}}

\DeclareMathOperator{\TLIM}{TLIM}
\DeclareMathOperator{\LIM}{LIM}

\DeclarePairedDelimiter{\paren}{(}{)}
\DeclarePairedDelimiter{\set}{\{}{\}}
\DeclarePairedDelimiter{\brac}{[}{]}
\DeclarePairedDelimiter{\norm}{\|}{\|}
\DeclarePairedDelimiter{\net}{\langle}{\rangle}
\DeclarePairedDelimiter{\seq}{\langle}{\rangle_{n=1}^{\infty}}
\DeclarePairedDelimiter{\abs}{|}{|}

\theoremstyle{definition}
\newtheorem{Theorem}{Theorem}[section]
\newtheorem{Lemma}[Theorem]{Lemma}

\newtheorem{Proposition}[Theorem]{Proposition}

\newtheorem{Question}[Theorem]{Question}
\newtheorem{Definition}[Theorem]{Definition}
\newtheorem{point}[Theorem]{}

\begin{document}
\title{When is an invariant mean the limit of a F\o{}lner net?}
\author{John Hopfensperger}
\thanks{The present paper will form part of the author's PhD thesis under the supervision of Ching Chou.}
\address{\hskip-\parindent
Department of Mathematics, University at Buffalo,
Buffalo, NY 14260-2900, USA}
\email{johnhopf@buffalo.edu}


\begin{abstract}
Let $G$ be a locally compact amenable group, $\TLIM(G)$
	the topological left-invariant means on $G$,
	and $\TLIM_0(G)$ the limit points of F\o{}lner-nets.
I show that $\TLIM_0(G) = \TLIM(G)$ unless $G$ is $\sigma$-compact non-unimodular,
	in which case $\TLIM_0(G) \neq \TLIM(G)$.
This improves a 1970 result of Chou and a 2009 result of Hindman and Strauss.
I consider the analogous problem for the non-topological left-invariant means,
	and give a short construction of a net converging to invariance ``weakly but not strongly,''
	simplifying the proof of a 2001 result of Rosenblatt and Willis.
\end{abstract}
\maketitle

\section{History}
In this paper, \(G\) is always a locally compact group.
The left Haar measure of \(E\subset G\) is denoted \(|E|\).
The set of means on \(\Linfty\) is \(\Mcal(G) = \{\mu\in\LinftyS : \|\mu\| = 1,\ \mu\geq 0\}\),
	which is endowed with the \(w^*\)-topology to make it compact.
Regarding \(\Lone\) as a subset of \(\LinftyS\),
define \(\Mcal_1(G) = \set*{f\in\Lone : \|f\|_1 = 1,\ f\geq 0}\).

\begin{Proposition}[{\cite[page 1]{greenleaf}}]\label{Prop M1 Dense in M}
\(\Mcal_1(G)\) is dense in \(\Mcal(G)\).
\end{Proposition}

Define left-translation of functions by \(l_x\phi(y) = \phi(x^{-1}y)\).
The set of left-invariant means on \(\Linfty\) is
\begin{center}
\(\LIM(G) =
\set*{\mu\in\Mcal(G) : \paren*{\forall x\in G}\ \paren*{\forall \phi\in\Linfty}\ 
	\mu(\phi) = \mu(l_x \phi)}\).
\end{center}

\(G\) is said to be amenable if \(\LIM(G)\) is nonempty.
Proposition~\ref{Prop M1 Dense in M} shows that every left-invariant mean \(\mu\) is the limit
	of a net \(\net*{f_\alpha}\) in \(\Mcal_1(G)\).
Such a net is said to converge weakly to invariance, because
the net \(\net*{f_\alpha - l_x f_\alpha}\) converges weakly to $0$ for all \(x\in G\).

A mean \(f\in\Mcal_1(G)\) is said to be \((K,\epsilon)\)-invariant if
\(\norm*{f - l_x f}_1 < \epsilon\) for each \(x\in K\).
A net in \(\Mcal_1(G)\) is said to converge strongly to invariance if it is eventually
\((K,\epsilon)\)-invariant for each finite \(K\subset G\) and \(\epsilon > 0\).

\begin{Proposition}[{\cite[Theorem 2.4.2]{greenleaf}}]
If $G$ is amenable, \(\Mcal_1(G)\) contains a net converging strongly to invariance.
\end{Proposition}

Let \(\Cp\) be the compact subsets of \(G\) with positive measure.
For \(A\in\Cp\) define \(\mu\sub{A} = \chi\sub{A} / |A| \in \Mcal_1(G)\).
If \(\mu\sub{A}\) is \((K,\epsilon)\)-invariant, we may also say
	\(A\) is \((K,\epsilon)\)-invariant.
Notice \(\norm*{\mu\sub{A} - l_x \mu\sub{A}}_1 = \abs*{A \Sdiff xA}/{\abs*{A}}\).
Let \(\Mcal_{\Cp}(G) = \set*{\mu\sub{A} : A\in \Cp}\).

\begin{Proposition}[{\cite[Theorem 3.6.1]{greenleaf}}]
If $G$ is amenable, \(\Mcal_{\Cp}(G)\) contains a net \(\net{\mu\sub{A_\alpha}}\) converging strongly to invariance.
In this case, both \(\net{\mu\sub{A_\alpha}}\) and \(\net*{A_\alpha}\) are called F\o{}lner nets.
\end{Proposition}

\begin{Question}
Is it possible to construct a net converging weakly but not strongly to invariance?

The main result of \cite{WeakNotStrong} is that every \(\mu\in\LIM(G)\)
is the limit of some net in \(\Mcal_1(G)\) \textit{not} converging strongly to invariance.
Theorems~\ref{Theorem weak but not strong discrete}
and \ref{Theorem weak but not strong nondiscrete}
give a shorter proof of a slightly stronger result:
Every non-atomic \(\mu\in\Mcal(G)\) is the limit of some net in \(\Mcal_{\Cp}(G)\) not converging strongly to invariance.
\end{Question}

\begin{Question}\label{Question LIM0 = LIM}
Let \(\LIM_0(G)\) be the limit points of F\o{}lner nets.
Does \(\LIM_0(G) = \LIM(G)\)?

Hindman and Strauss asked this question for amenable semigroups in \cite{HindmanDensity},
and answered it affirmatively for the semigroup \((\Nbb, +)\).
The following theorems extend their affirmative result:
Theorem~\ref{Theorem Unimodular Case} when \(G\) is discrete,
Theorem~\ref{Theorem LIM when G is large} when \(G\) is larger than \(\sigma\)-compact, and
Theorem~\ref{Theorem LIM0} when \(G\) is non-discrete but amenable-as-discrete.
The question remains open when \(G\) is compact or \(\sigma\)-compact but not amenable-as-discrete.
\end{Question}

For \(f\in\Mcal_1(G)\), regard \(f*\phi\) as the average
	\(\int_G f(x)\: l_x \phi \dd{x}\) of left-translates of \(\phi\).
The set of topological left-invariant means on \(\Linfty\) is
\begin{center}\(\TLIM(G) = \set*{\mu\in\Mcal(G) : \paren*{\forall f\in\Mcal_1(G)}\ 
\paren*{\forall \phi\in\Linfty}\ \mu(\phi) = \mu(f*\phi)}\).\end{center}

\begin{Proposition}
When $G$ is discrete, \(\LIM(G) = \TLIM(G)\).
\end{Proposition}
\begin{proof}
When $G$ is discrete, every $f\in\Mcal_1(G)$ is a sum of point-masses.
Convolution by a point-mass is equivalent to left-translation.
\end{proof}

A net in \(\Mcal_1(G)\) is said to converge strongly to topological invariance if it is eventually
	\((K,\epsilon)\)-invariant for each compact \(K\subset G\) and \(\epsilon > 0\).

\begin{Proposition}[{\cite[Corollary 2.4.4]{greenleaf}}]
	If \(\net*{f_\alpha}\) is a net in \(\Mcal_1(G)\) converging strongly to topological invariance, and \(f_\alpha \to \mu\), then \(\mu\in\TLIM(G)\).
\end{Proposition}

\begin{Proposition}[{\cite[Proposition 3.6.2]{greenleaf}}]
If $G$ is amenable, \(\Mcal_{\Cp}(G)\) contains a net \(\net*{\mu\sub{A_\alpha}}\)
	converging strongly to topological invariance.
In this case, both \(\net*{\mu\sub{A_\alpha}}\) and \(\net*{A_\alpha}\)
	are called topological F\o{}lner nets.
\end{Proposition}

For example, \(\seq*{[-n, n]}\) is a topological F\o{}lner sequence for \(\Rbb\).
By \cite[Theorem 3]{EmersionRatio}, there is no such thing as a non-topological F\o{}lner sequence.

\begin{Question}\label{Question non topological Folner net}
If $G$ is amenable and non-discrete, is it possible to construct a non-topological F\o{}lner net?

When $G$ is amenable-as-discrete, there exists \(\mu \in \LIM(G) \setminus \TLIM(G)\),
\cf\ \cite{Rosenblatt76}.
In this case, Theorem~\ref{Theorem LIM0} yields
a net \(\net*{\mu\sub{X_\alpha S_\alpha}}\) converging to \(\mu\),
which is clearly a non-topological F\o{}lner net.
Otherwise the question remains open.
\end{Question}

\begin{Question}
Let \(\TLIM_0(G)\) be the limit points of topological F\o{}lner nets.
Does \(\TLIM_0(G) = \TLIM(G)\)?

A partial answer was given by Chou, who showed in \cite{Chou70} that
	\(\TLIM(G)\) is the closed convex hull of \(\TLIM_0(G)\).
The main result of the present paper is to answer this question completely.
If $G$ is $\sigma$-compact non-unimodular,
	the answer is negative by Theorem~\ref{Theorem Non-Unimodular}.
In all other cases the answer is affirmative:
by Proposition~\ref{Prop When G is compact} when \(G\) is compact,
by Theorem~\ref{Theorem Unimodular Case} when \(G\) is unimodular,
and by Theorem~\ref{Theorem when G is large} when \(G\) is larger than \(\sigma\)-compact.
\end{Question}


\section{Converging to invariance weakly but not strongly}
A neighborhood basis about $\mu\in\Mcal(G)$ is given by sets of the form
\begin{center}
\(\Ncal(\mu, \Fcal, \epsilon) =
\set*{ \nu\in\Mcal(G) :
	\paren*{\forall f\in\Fcal}\
	\abs*{\mu(f) - \nu(f)} < \epsilon
	}\)
\end{center}
where $\Fcal$ ranges over finite subsets of $\Linfty$ and $\epsilon$ ranges over $(0, 1)$.

\begin{Lemma}\label{Lemma Partition Basis}
Regard each $\mu\in\Mcal$ as a finitely additive measure via \(\mu(E) = \mu(\chi_E)\).
Then a neighborhood basis about $\mu\in\Mcal$ is given by sets of the form
\begin{center}
\(\Ncal(\mu, \Pcal, \epsilon) = \set*{
	\nu\in\Mcal :
	\paren*{\forall E\in\Pcal}\
	|\mu(E) - \nu(E)| < \epsilon}\)
\end{center}
where $\epsilon$ ranges over $(0,1)$ and \(\Pcal\) ranges over finite measurable partitions of $G$.
\end{Lemma}
\begin{proof}
Pick $\Ncal(\mu, \Fcal, \epsilon)$.
For simplicity, suppose $\Fcal$ consists of simple functions.
Let $M = \max\set*{\|f\|_\infty : f\in \Fcal}$.
Let $\Pcal$ be the atoms of the measure algebra generated by $\Fcal$.
Then $\Ncal\big( \mu, \Pcal, \mfrac{\epsilon}{\#\Pcal\cdot M} \big)
	\subset \Ncal(\mu, \Fcal, \epsilon)$.
\end{proof}

\begin{Lemma}\label{Lemma if E is infinite}
Let $E\subset G$ be any infinite subset, $n\in\Nbb$, and $x\in G\setminus \{e\}$.
Then there exists $S\subset E$ with $\#S = n$ and $S \cap xS = \varnothing$.
\end{Lemma}
\begin{proof}
If $n=0$, take \(S = \varnothing\).
Inductively, suppose there exists $R \subset E$ with $\# R = n-1$ and $R \cap xR = \varnothing$.
Pick any $y\in E \setminus \set{x, x^{-1}}R$, and let $S = R \cup \{y\}$.
\end{proof}

\begin{Theorem}\label{Theorem weak but not strong discrete}
Suppose $G$ is discrete, $\mu\in\Mcal$ vanishes on finite sets, and $x\in G\setminus \{e\}$.
Then there exists a net \(\net*{S_\Pcal}\) in \(\Cp\) so \(\mu\sub{S_\Pcal} \to \mu\),
	but \(S_\Pcal \cap x S_\Pcal = \varnothing\) for all \(\Pcal\).
\end{Theorem}
\begin{proof}
Let $\Pscr$ be the directed set of all finite partitions of $G$, ordered by refinement.
Pick \(\Pcal = \set*{E_1, \hdots, E_p} \in \Pscr\).
For $1\leq i\leq p$, we will choose $S_i\subset E_i$ such that
\(\abs*{\mfrac{\# S_i}{\# S_1 + \hdots + \# S_p} - \mu(E_i)} < \mfrac1p\),
	and take \(S_\Pcal = S_1 \cup \hdots \cup S_p\).
Then \(\abs*{\mu\sub{S_\Pcal}(E_i) - \mu(E_i)} < \mfrac1p\),
	and \(\mu\sub{S_\Pcal}\to\mu\) by Lemma~\ref{Lemma Partition Basis}.

We begin by establishing the values \(n_i = \# S_i\).
If $\mu(E_i) = 0$, let \(n_i = 0\).
Otherwise $\mu(E_i) > 0$, hence $E_i$ is infinite.
In this case, let $n_i\geq 0$ be an integer such that
	\(\abs*{\mfrac{n_i}{2p^2} - \mu(E_i)} < \mfrac{1}{2p^2}\).
Let \(N = \sum_{i=1}^p n_i\).
Now \(\abs*{\mfrac{n_i}{2p^2} - \mfrac{n_i}{N}}
	= \mfrac{n_i}{N} \cdot \abs*{\mfrac{N}{2p^2} - 1}
	\leq \abs*{\mfrac{N}{2p^2} - 1}
	= \abs*{\sum_i \mfrac{n_i}{2p^2} - \mu(E_i)}
	< \mfrac{1}{2p}\),
so \(\abs*{\mu(E_i) - \mfrac{n_i}{N}} < \mfrac1p\).

Apply Lemma~\ref{Lemma if E is infinite} to choose \(S_1 \subset E_1\)
	with \(S_1 \cap x S_1 = \varnothing\).
For \(k <p\), inductively choose
	\(S_{k+1} \subset E_{k+1} \setminus \set{x, x^{-1}} (S_1 \cup \hdots \cup S_k)\)
	with \(\# S_{k+1} = n_k\) and \(S_{k+1} \cap x S_{k+1} = \varnothing\).
Let $S_\Pcal = S_1 \cup \hdots \cup S_p$.
Now $\mu\sub{S_\Pcal}(E_i) = \mfrac{n_i}{N}$ and $S_\Pcal \cap xS_\Pcal = \varnothing$, as desired.
\end{proof}

\begin{point}
The hypothesis ``\(\mu\) vanishes on finite sets'' is necessary in Theorem~\ref{Theorem weak but not strong discrete}.
For example, pick \(x, y\in G\) and define \(\mu\in\Mcal\) by
	\(\mu(\{x\}) = \frac23\) and \(\mu(\{y\}) = \frac13\).
For each \(F\in\Cp\), \(\mu\sub{F}(\{x\}) \in \set*{1, \frac12, \frac13, \hdots}\),
hence \(\abs*{\mu(\{x\}) - \mu\sub{F}(\{x\})} \geq \frac16\).
This foreshadows Theorem~\ref{Theorem Non-Unimodular}.
\end{point}

\begin{Lemma}\label{Lemma Very Small S}
Suppose $G$ is not discrete.
Pick $E\subset G$ with positive measure, and $X = \{x_1, \hdots, x_n\} \subset G$.
For any \(c > 0\), there exists $S \subset E$ such that $0 < |S| \leq c$
	and $\{x_1 S, \hdots, x_n S\}$ are mutually disjoint.
\end{Lemma}
\begin{proof}
Let $K\subset E$ be any compact set with positive measure.
Let $U$ be a small neighborhood of $e$, so $U U^{-1} \cap X^{-1} X = \{e\}$
	and $\max_{k\in K}|Uk|\leq c$.
Pick \(k_1, \hdots, k_n\in K\) such that \(K\subset U k_1 \cup \hdots \cup U k_n\).
Now $0 < |K| \leq \sum_{i=1}^n |U k_i \cap K|$, hence $0 < |U k_i \cap K|$ for some $i$.
Take $S = U k_i \cap K$.
\end{proof}

\begin{Theorem}\label{Theorem weak but not strong nondiscrete}
Suppose $G$ is not discrete.
Given any \(\mu\in\Mcal\) and \(x\in G\setminus\{e\}\),
	there exists a net \(\net*{S_\Pcal}\) in \(\Cp\)
	so that \(\mu\sub{S_\Pcal}\to \mu\), but
	\(S_\Pcal \cap x S_\Pcal = \varnothing\) for all \(\Pcal\).
\end{Theorem}
\begin{proof}
The following construction yields sets \(\net*{S_\Pcal}\) that may not be compact.
This suffices to prove the theorem, since each \(S_\Pcal\)
can be approximated from within by a compact set.

Let \(\Pscr\) be the directed set of all finite measurable partitions of \(G\),
	ordered by refinement.
Pick \(\Pcal=\{E_1, \hdots, E_p\}\in \Pscr\).
For \(1\leq i\leq p\), we will choose \(S_i'\subset E_i\) such that
$\abs*{S_i'} / \big(\abs*{S_1'} + \hdots + \abs*{S_k'}\big) = \mu(E_i)$,
	then take $S_\Pcal = S_1' \cup \hdots \cup S_p'$.
Thus $\mu\sub{S_\Pcal}(E_i) = \mu(E_i)$,
	and $\mu\sub{S_\Pcal}\to\mu$ by Lemma~\ref{Lemma Partition Basis}.

If $\mu(E_i) = 0$ we can take $S_i = \varnothing$, so assume
$0 < m=\min\{1, |E_1|,\hdots,|E_p|\}$ and let $c = m / 2p$.
By Lemma~\ref{Lemma Very Small S},
	choose \(S_1 \subset E_1\) with \(0 <|S_1| \leq c\) and \(S_1 \cap xS_1 = \varnothing\).
For \(k < p\), inductively choose
\(S_{k+1} \subset E_{k+1} \setminus \set{x, x^{-1}}(S_1 \cup \hdots \cup S_k)\)
with \(0<|S_{k+1}| \leq c\) and \(S_{k+1} \cap xS_{k+1} = \varnothing\).
This is possible, since \(|E_{k+1} \setminus \set{x, x^{-1}}(S_1 \cup \hdots \cup S_k)|
	\geq m - 2kc > 0\).

Finally, let \(m' = \min\set*{|S_1|, |S_2|, \hdots, |S_p|}\).
For each $i$, choose \(S_i' \subset S_i\) with \(|S_i'| = m'\cdot \mu(E_i)\).
Now $\mu\sub{S_\Pcal}(E_i) = \mu(E_i)$ and \(S_\Pcal \cap xS_\Pcal = \varnothing\), as desired.
\end{proof}


\section{\texorpdfstring{$\bm{\kappa}$-Compactness}{kappa-Compactness}}

\begin{Definition}
For the rest of the paper, triples of the form \((\Pcal, K, \epsilon)\)
	are always understood to range over
	finite measurable partitions \(\Pcal\) of \(G\),
	compact sets \(K\subset G\), and \(\epsilon \in (0,1)\).
Recall that \(\Cp(K,\epsilon)\) is the set of all compact
	\((K,\epsilon)\)-invariant sets with positive measure.
In these terms, we can give the formal definition:
\begin{center}
\(\TLIM_0(G) = \set*{
	\mu\in\TLIM(G) : \paren*{\forall \paren*{\Pcal, K, \epsilon}}\ 
	\paren*{\exists A\in\Cp(K,\epsilon)}\ 
	\mu\sub{A} \in \Ncal(\mu, \Pcal, \epsilon)}\).
\end{center}
\end{Definition}

For $S\subset G$, let $\kappa(S)$ denote the smallest cardinal such that there exists $\Kcal$,
a collection of compact subsets of $G$ with $\#{\Kcal} = \kappa(S)$ and $S\subset \bigcup \Kcal$.
Notice that either \(\kappa(G) = 1\), \(\kappa(G) = \Nbb\), or \(\kappa(G) > \Nbb\).

\begin{Proposition}\label{Prop When G is compact}
When $\kappa(G) = 1$, $\TLIM_0(G) = \TLIM(G) = \{\mu\sub{G}\}$.
\end{Proposition}
\begin{proof}
Pick $\mu\in\TLIM(G)$.
Since $G$ is compact, $\chi_G\in C_0(G)$ and \(\mu\sub{G}(\chi_G) = 1\).
Thus $\mu|\sub{C_0(G)}$ is nonzero, and it induces a nonzero left-invariant measure on $G$
via the Riesz-Kakutani representation theorem.
By the uniqueness of Haar measure, we see $\mu|\sub{C_0(G)} = \mu\sub{G}|\sub{C_0(G)}$.
Pick $\phi\in\Linfty$ and $f\in P_1(G)$.
Then $f*\phi\in C_0(G)$, so $\mu(\phi) = \mu(f*\phi) = \mu\sub{G}(f*\phi) = \mu\sub{G}(\phi)$.
\end{proof}

\begin{Lemma}\label{Lemma if Kappa S is small then measure 0}
If $\mu\in\LIM(G)$ and $\kappa(S) < \kappa(G)$, then $\mu(S) = 0$.
\end{Lemma}
\begin{proof}
Since \(\Nbb \cdot \kappa(S) \leq \kappa(G)\), we can find disjoint translates
\(\set*{x_1 S, x_2 S, \hdots}\).
If $\mu(S) > 0$, then $\mu(x_1 S \cup \hdots \cup x_n S) = n \cdot \mu(S)$
is eventually greater than 1, contradicting \(\norm*{\mu} = 1\).
\end{proof}

\begin{point}
Throughout the paper, we make free use of the following formulas:
\\For \(A,B \in\Cp\),
\(\norm*{\mu\sub{A} - \mu\sub{B}}_1
	= \mfrac{|A\setminus B|}{|A|} + \mfrac{|B\setminus A|}{|B|}
	+ |A\cap B| \cdot \abs*{\mfrac{1}{|A|} - \mfrac{1}{|B|}}\).
\\If \(A\subset B\),
this becomes
\(\norm*{\mu\sub{A} - \mu\sub{B}}_1
	= 0
	+ \mfrac{|B\setminus A|}{|B|}
	+ |A| \paren*{\mfrac{1}{|A|} - \mfrac{1}{|B|}}
	= 2 \mfrac{|B\setminus A|}{|B|}.\)
\end{point}

\begin{Lemma}\label{Lemma many disjoint sets}
Pick \((\Pcal, K, \epsilon)\) and $\mu\in\TLIM_0(G)$.
Let \(\kappa = \kappa(G)\).
There exists a family of mutually disjoint sets
\(\set*{A_\alpha : \alpha < \kappa} \subset \Cp(K,\epsilon)\)
such that \(\set*{\mu\sub{A_\alpha} : \alpha < \kappa} \subset \Ncal(\mu, \Pcal, \epsilon)\).
\end{Lemma}
\begin{proof}
Choose any precompact open set \(U\).
For \(\beta < \kappa(G)\),
suppose \(\set*{A_\alpha : \alpha < \beta} \subset \Cp(K,\epsilon)\)
have been chosen so that \(\set*{A_\alpha U : \alpha < \beta}\) are mutually disjoint
and \(\set*{\mu\sub{A_\alpha} : \alpha < \beta} \subset \Ncal(\mu, \Pcal, \epsilon)\).
Once \(A_\beta\) is constructed, the result follows by transfinite induction.

Define \(B = \bigcup_{\alpha < \beta} A_\alpha UU^{-1}\), which is open and thus measurable.
Notice \(\kappa(B) = \beta < \kappa(G)\).
If \(\Pcal = \set*{E_1, \hdots, E_p}\),
let $\Pcal_B = (E_1 \setminus B, \hdots, E_p \setminus B, B)$.
Let \(\delta = \epsilon/4 \),
and pick \(A\in\Cp(K,\delta)\) with \(\mu\sub{A} \in \Ncal(\mu, \Pcal_B, \delta)\).
Define \(A_\beta = A \setminus B\),
which ensures \(\set*{A_\alpha U : \alpha\leq \beta}\) are mutually disjoint.
Lemma~\ref{Lemma if Kappa S is small then measure 0} tells us $\mu(B) = 0$,
hence \(\mu\sub{A}(B) < \delta\).
Thus \(\big\|\mu\sub{A} - \mu\sub{A_\beta}\big\|_1
	= 2 \mfrac{|A \setminus A_\beta|}{|A|}
	= 2\mfrac{|A\cap B|}{|A|}
	= 2 \mu\sub{A}(B)
	< 2 \delta\).
By the triangle inequality, \(\mu\sub{A_\beta} \in \Ncal(\mu, \Pcal, 3\delta)\).
If \(x\in K\), \(\abs{xA_\beta \Sdiff A_\beta}
	\leq \abs{xA_\beta \Sdiff xA} + \abs{xA \Sdiff A} + \abs{A\Sdiff A_\beta}
	<3\delta \abs{A} < 4\delta \abs{A_\beta}\),
so \(A_\beta\in\Cp(K,4\delta)\).
\end{proof}


\section{The method of Hindman and Strauss}

\begin{Lemma}\label{Lemma TLIM0 is total in TLIM}
The closed convex hull of $\TLIM_0(G)$ is all of $\TLIM(G)$.
\end{Lemma}
\begin{proof}
Chou originally proved this for $\sigma$-compact groups, see \cite[Theorem 3.2]{Chou70}.
In \cite{milnes}, Milnes points out that the result is valid even when $G$ is not $\sigma$-compact,
although his construction of a F\o{}lner-net \(\net*{U_\alpha}\) has a small problem:
For each index \(\alpha\), he asks us to choose a compact set \(U_\alpha\),
such that \(U_\beta \subset U_\alpha\) for \(\beta \prec \alpha\).
However, \(\alpha\) may have infinitely many predecessors,
in which case \(\bigcup_{\beta \prec \alpha} U_\beta\) has no reason to be precompact!
For a correct proof in full generality, see \cite[Lemma 4.11]{hopfensperger2020}.
\end{proof}

The following deceptively simple lemma is due to Hindman and Strauss,
see \cite[Proof of Theorem 4.5]{HindmanDensity}.
\begin{Lemma}\label{Main Idea}
Suppose \(\brac*{\mu,\nu\in\TLIM_0(G)} \Rightarrow \brac*{\frac12(\mu + \nu)\in \TLIM_0(G)}\).
Then \(\TLIM_0(G) = \TLIM(G)\).
\end{Lemma}
\begin{proof}
Since the dyadic rationals are dense in \([0,1]\) and \(\TLIM_0(G)\) is closed,
the hypothesis implies \(\TLIM_0(G)\) is convex.
The result follows by Lemma~\ref{Lemma TLIM0 is total in TLIM}.
\end{proof}

Lemma~\ref{Main Idea} is useful because
\(\frac12\paren*{\mu\sub{A} + \mu\sub{B}} = \mu\sub{A\cup B}\)
when \(A,B \in \Cp\) are disjoint and equal in measure.
Lemma~\ref{Lemma Mean Approximation} tells us what happens when \(A\) and \(B\)
are \textit{approximately} disjoint and equal in measure.

\begin{Lemma}\label{Lemma Mean Approximation}
Pick \(A, B \in\Cp\) and \(\delta \in \bigl(0, \frac12\bigr)\).
Suppose \(\mu\sub{B}(A) < \delta\) and \(|A|/|B| \in \bigl( (1-\delta)^2, (1+\delta)^2\bigr)\).
Then \(\norm*{\mu\sub{A\cup B} - \frac12(\mu\sub{A} + \mu\sub{B})}_1 < 3\delta\).
\end{Lemma}
\begin{proof}
Let $B' = B\setminus A$ and $r = \mfrac{|A|}{|B'|} = \mfrac{|A|}{|B|} \cdot \mfrac{|B|}{|B'|}$.
Since \(1\geq \mfrac{|B'|}{|B|} = \mfrac{|B| - |B\cap A|}{|B|} = 1 - \mu\sub{B}(A) > 1 - \delta\),
we see \( r\in \bigl((1-\delta)^2, \frac{(1 + \delta)^2}{1 - \delta}\bigr)\).
In these terms,
\(\mu\sub{A\cup B} = \mu\sub{A \cup B'}
	= \mfrac{|A|}{|A|+|B'|} \mu\sub{A} + \mfrac{|B'|}{|A|+|B'|}\mu\sub{B'}
	= \paren*{\mfrac{1}{1+r^{-1}}}\mu\sub{A} + \paren*{\mfrac{1}{1 + r}} \mu\sub{B'}\).
Now we can compute
\(\norm*{ \mu\sub{A\cup B} - \frac12 \big(\mu\sub{A} + \mu\sub{B'}\big) }_1
	\leq \abs*{ \mfrac12 - \mfrac{1}{1 + r^{-1}} } \cdot \norm*{\mu\sub{A}}_1
		+ \abs*{\mfrac12 - \mfrac{1}{1 + r} } \cdot \norm*{\mu\sub{B'}}
	= \abs*{\mfrac{1-r}{1+r}} < 2\delta\).
On the other hand,
\(\big\|\mu\sub{B} - \mu\sub{B'}\big\|_1
	= 2\mfrac{|B \setminus B'|}{|B|}
	< 2\delta\),
so \(\norm*{\frac12 \paren*{\mu\sub{A} + \mu\sub{B'}}
	- \frac12 \paren*{\mu\sub{A} + \mu\sub{B}}}_1 < \delta\).
The result follows by the triangle inequality.
\end{proof}

\begin{Lemma}\label{Lemma Invariance of Union}
Suppose $A,B \in \Cp(K, \delta)$. Then $A\cup B\in \Cp(K, 2\delta)$.
If $A\cap B = \varnothing$, then $A\cup B\in \Cp(K, \delta)$.
\end{Lemma}
\begin{proof}
Pick \(x\in K\).
Notice $x(A\cup B) \Sdiff (A \cup B) \subset (xA \Sdiff A) \cup (xB \Sdiff B)$.
\\It follows that \(\mfrac{|x(A\cup B) \Sdiff (A\cup B)|}{|A\cup B|}
	\leq \mfrac{|xA\Sdiff A|}{|A\cup B|} + \mfrac{|xB\Sdiff B|}{|A\cup B|}
	< \mfrac{\delta|A|}{|A\cup B|} + \mfrac{\delta|B|}{|A\cup B|}
	< 2\delta\).
\\If \(A\cap B = \varnothing\),
then \(|A\cup B| = |A| + |B|\),
hence $\mfrac{\delta|A|}{|A\cup B|} + \mfrac{\delta|B|}{|A\cup B|} = \delta$.
\end{proof}

\begin{Theorem}\label{Theorem Hindman and Strauss Technique}
Let \(G\) be noncompact.
Then the following statement implies \(\TLIM_0(G) = \TLIM(G)\):
\begin{center}
	\(\paren*{\forall (p, K, \epsilon)}\ 
	\paren*{\exists M > 0}\ 
	\paren*{\forall \mu\in\TLIM_0(G)}\ 
	\paren*{\forall \Pcal\ \text{with}\ \#\Pcal=p}\)
	\\\(\exists A\in\Cp(K,\epsilon)
	\text{ with } |A|/M \in (1-\epsilon, 1+\epsilon)
	\text{ and } \mu\sub{A}\in\Ncal(\mu, \Pcal, \epsilon)\).
\end{center}

The crux of the statement is that \(M\) is allowed to depend on
\((\#\Pcal, K, \epsilon)\), but not on \(\Pcal\) itself.
\end{Theorem}
\begin{proof}
Pick $(\Pcal, K, \epsilon)$ and $\mu,\nu \in \TLIM_0(G)$.
Say \(\Pcal = \paren*{E_1, \hdots, E_p}\) and \(\delta = \epsilon/4\).
By hypothesis, obtain \(M>0\) for \( (p+1, K, \delta)\).
Let \(\Pcal_{\varnothing} = \paren*{E_1, \hdots, E_p, \varnothing}\),
	so \(\#\Pcal_\varnothing = p+1\).
Pick \(A\in\Cp(K,\delta)\) with \(\mu\sub{A}\in\Ncal(\mu, \Pcal_\varnothing, \delta)\)
	and \(|A|/M \in (1-\delta, 1+\delta)\).
Let \(\Pcal_A = (E_1 \setminus A, \hdots, E_p \setminus A, A)\).
Pick \(B\in\Cp(K,\delta)\)
	with \(\mu\sub{B} \in \Ncal\paren{\nu, \Pcal_A, \delta}\)
	and \(|B|/M\in (1-\delta, 1+\delta)\).
Now \(|A| / |B| \in \bigl((1-\delta)^2, (1+\delta)^2\bigr)\).
Since \(A\) is compact but \(G\) is not,
Lemma~\ref{Lemma if Kappa S is small then measure 0} tells us
\(\nu(A) = 0\), hence \(\mu\sub{B}(A) < \delta\).
Applying Lemma~\ref{Lemma Mean Approximation} to the previous two statements, we conclude
\(\norm*{\frac12\paren*{\mu\sub{A} + \mu\sub{B}} - \mu\sub{A\cup B}}_1 < 3\delta\).
By the triangle inequality,
\(\mu\sub{A\cup B} \in \Ncal\big(\frac12(\mu+\nu), \Pcal, 4\delta\big)\).
By Lemma~\ref{Lemma Invariance of Union}, \(A\cup B \in \Cp(K,2\delta)\).
We conclude \(\frac12(\mu + \nu)\in\TLIM_0(G)\).
By Lemma~\ref{Main Idea}, we are done.
\end{proof}

\section{Ornstein-Weiss quasi-tiling}

\begin{point}\label{Ci definitions}
We require the following notions of \((K,\epsilon)\)-invariance.
\\\phantom{ }
	\(\Cp_0(K,\epsilon) = \set*{A\subset G :
	A\text{ is compact and }\ (\forall x\in K)\ |xA\Sdiff A| / |A| < \epsilon}\).
This is just \(\Cp(K,\epsilon)\) above.
\\\phantom{ }
\(\Cp_1(K,\epsilon) = \{A\subset G : A \text{ is compact and }
	|KA \Sdiff A| / |A| < \epsilon\}\).
\\\phantom{ }
	\(\Cp_2(K,\epsilon) = \set*{A\subset G : A \text{ is compact and }
	|\partial_K(A)| / |A| < \epsilon}\),
	where \(\partial_K(A) = K A \setminus \bigcap_{x\in K} x A\).
\\This differs from \cite{OW87},
where ``the \(K\)-boundary of \(A\)''
is defined as \(K^{-1} A \setminus \bigcap_{x\in K^{-1}} x A\).

For \(j\in\set{0,1,2}\), define
\begin{center}
\(\TLIM_j(G) = \big\{
	\mu\in\TLIM(G) : \paren*{\forall \paren*{\Pcal, K, \epsilon}}\ 
	\big(\exists A\in\Cp_j(K,\epsilon) \big)\ 
	\mu\sub{A} \in \Ncal(\mu, \Pcal, \epsilon)\big\}\).
\end{center}

We shall say ``\(\Cp_i\) is asymptotically contained in \(\Cp_j\)'' to mean the following:
\begin{center}
\(\paren*{\forall(K,\epsilon)}\
	\paren*{\exists(K',\epsilon')}\
	\paren*{\forall A\in\Cp_i(K', \epsilon')}\
	\exists B\in\Cp_j(K,\epsilon) \text{ with } \norm*{\mu\sub{A}-\mu\sub{B}}_1 < \epsilon\).
\end{center}
\end{point}

\begin{Lemma}
To prove \(\TLIM_i(G) \subset \TLIM_j(G)\), it suffices to show
\(\Cp_i\) is asymptotically contained in \(\Cp_j\).
\end{Lemma}
\begin{proof}
Pick \((\Pcal, K, \epsilon)\) and \(\mu\in\TLIM_i(G)\).
By hypothesis, obtain \((K', \epsilon')\).
Obtain \(A\in\Cp_i(K', \epsilon')\) with \(\mu\sub{A}\in\Ncal(\mu, \Pcal, \epsilon)\).
Obtain \(B\in\Cp_j(K, \epsilon)\) with \(\norm*{\mu\sub{A} - \mu\sub{B}}_1 < \epsilon\).
By the triangle inequality, \(\mu\sub{B}\in\Ncal(\mu, \Pcal, 2\epsilon)\).
Since \(\epsilon\) was arbitrary, we conclude \(\mu\in\TLIM_j(G)\).
\end{proof}

\begin{Lemma}\label{Lemma no TLIMi left behind}
\(\TLIM_0(G) = \TLIM_1(G) = \TLIM_2(G).\)
\end{Lemma}
\begin{proof}
Trivially, \(\Cp_2(K,\epsilon)\subset \Cp_1(K,\epsilon)\).
To prove \(\Cp_1(K, \epsilon)\subset \Cp_0(K, 2\epsilon)\), suppose \(x\in K\) and \(A\in\Cp_1(K,\epsilon)\).
Then \(|xA \Sdiff A|  = 2|xA \setminus A| \leq 2|KA \setminus A| < 2\epsilon|A|\).
\cite[Theorem 15]{EmersionRatio} shows that \(\Cp_0\) is asymptotically contained in \(\Cp_1\).
To prove \(\Cp_1\) is asymptotically contained in \(\Cp_2\),
pick \((K,\epsilon)\) and let \(J = K \cup K^{-1} \cup \set{e}\).
Suppose \(A\in\Cp_1\paren*{J^2, \epsilon}\), and let \(B = JA\).
Then \(A \subset \partial_{J}(B)\), and \(JB \setminus \partial_J(B) \subset J^2 A \setminus A\), which shows \(B\in\Cp_2(J, \epsilon) \subset \Cp_2(K, \epsilon)\).
Also, since \(B \subset J^2 A\), we see \(|B| < (1+\epsilon)|A|\), hence
\(\norm*{\mu\sub{A} - \mu\sub{B}}_1
	= 2\mfrac{|A\setminus B|}{|A|}
	< 2\epsilon\).
\end{proof}

\begin{Lemma}\label{Lemma many disjoint sets 2}
Suppose \(G\) is noncompact, and pick \((\Pcal, K, \epsilon)\) and $\mu\in\TLIM_0(G)$.
There exists a family of mutually disjoint sets \(\set*{B_n: n\in\Nbb} \subset \Cp_1(K,\epsilon)\)
with \(\set*{\mu\sub{B_n} : n\in\Nbb} \subset \Ncal(\mu, \Pcal, \epsilon)\).
\end{Lemma}
\begin{proof}
By \cite[Theorem 15]{EmersionRatio}, obtain $(K', \epsilon')$ so
for each $A\in\Cp_0(K',\epsilon')$,
there exists \(B\in\Cp_1(K, \epsilon)\)
with \(\norm*{\mu\sub{A} - \mu\sub{B}}_1 < \epsilon/2\).
By Lemma~\ref{Lemma many disjoint sets},
choose a family of mutually disjoint sets \(\set*{A_n : n\in\Nbb}\subset\Cp_0(K', \epsilon')\)
with \(\set*{\mu\sub{A_n} : n\in\Nbb}\subset\Ncal\paren*{\mu,\Pcal, \epsilon/2}\).
For each \(n\), choose \(B_n \in \Cp_1(K, \epsilon)\)
with \(\norm*{\mu\sub{A_n} - \mu\sub{B_n}}_1 < \epsilon/2\).
By the triangle inequality, \(\mu\sub{B_n} \in \Ncal(\mu, \Pcal, \epsilon)\).
\end{proof}

\begin{Lemma}\label{Lemma A can be large}
Suppose $G$ is noncompact unimodular.
Pick $(\Pcal, K, \epsilon)$, $\mu\in\TLIM_0(G)$, and $M>0$.
Then there exists $B\in\Cp_2(K,\epsilon)$ with $|B| \geq M$
	and $\mu\sub{B}\in\Ncal(\mu,\Pcal,\epsilon)$.
\end{Lemma}
\begin{proof}
Without loss of generality, assume \(|K| > 0\).
Let \(\delta = \epsilon/3\).
By Lemma~\ref{Lemma many disjoint sets 2}, choose disjoint sets
$\set*{A_n : n\in\Nbb} \subset \Cp_1\paren*{K^2,\epsilon}$
with \(\set*{\mu\sub{A_n} : n\in\Nbb}\subset\Ncal\paren*{\mu,\Pcal,\delta}\).
For each \(n\), $|A_n| > (1-\epsilon)\abs*{K^2 A_n} > \abs*{K^2}$,
where the second inequality holds because \(G\) is unimodular.
For sufficiently large \(m\), \(\abs*{A_1 \cup \hdots \cup A_m} \geq M\).
Define \(A = A_1 \cup \hdots \cup A_m\).
Since this union is disjoint, it is easy to check \(A\in \Cp_1(K^2, \epsilon)\).
Notice \(\mu\sub{A}= \mfrac{|A_1|}{|A|} \mu\sub{A_1}
	+  \hdots + \mfrac{|A_m|}{|A|}\mu\sub{A_m} \in \Ncal\paren*{\mu,\Pcal,\delta}\),
as the latter set is convex.
Finally, let $B = KA$.
It follows, as in the proof of Lemma~\ref{Lemma no TLIMi left behind},
that \(B\in\Cp_2(K, \epsilon)\) and \(\norm*{\mu\sub{A} - \mu\sub{B}}_1 < 2\delta\),
hence \(\mu\sub{B}\in\Ncal(\mu,\Pcal, 3\delta)\) by the triangle inequality.
\end{proof}

Recall that \(\sqcup\) denotes a disjoint union.
\begin{Definition}\label{Definition Quasitiling}
Let \(A\subset G\) and \(T_1 \subset \hdots \subset T_n \subset G\) have finite positive measure.
\(\Tcal = \set*{T_1, \hdots, T_n}\) is said to
\textbf{\(\bm{\epsilon}\)-quasi-tile \(\bm{A}\)} if there exists
\(R = S_1 \cup \hdots \cup S_N = R_1 \sqcup \hdots \sqcup R_N\) satisfying the following:
\\\phantom{ }
(1) For \(i\in\set{1, \hdots, N}\), \(S_i = Tx\) for some $T\in\Tcal$ and $x\in G$.
\\\phantom{ }
(2) For \(i\in\set{1, \hdots, N}\), \(R_i \subset S_i\)
	with \(\norm*{\mu\sub{R_i} - \mu\sub{S_i}}_1 < \epsilon\).
\\\phantom{ }
(3) \(\norm*{\mu\sub{A} - \mu\sub{R}}_1 < \epsilon\).
\\\cite{OW87} gives the following weaker conditions in place of (2) and (3):
\\\phantom{ }
\((2')\) For \(i\in\set{1, \hdots, N}\), \(R_i \subset S_i\)
	with \(|R_i| > (1-\epsilon) |S_i|\).
\\\phantom{ }
\((3')\) \(R\subset A\) with \(|R| > (1-\epsilon) |A|\).
\\Of course \((2')\) implies \(\norm*{ \mu\sub{S_i} - \mu\sub{R_i} }_1
	= 2\mfrac{|S_i \setminus R_i|}{|S_i|}
	< 2 \epsilon\),
and \((3')\) implies \(\norm*{ \mu\sub{A} - \mu\sub{R} }_1
	= 2 \mfrac{|A\setminus R|}{|A|}
	< 2\epsilon\).
Therefore the two definitions become equivalent
when we quantify over all \(\epsilon \in (0,1)\),
as in Lemma~\ref{Ornstein Weiss Lemma}.
\end{Definition}

\begin{Lemma}\label{Ornstein Weiss Lemma}
Let \(G\) be unimodular. Pick \((K,\epsilon)\).
Then there exists \(\epsilon' >0\) and \(\Tcal = \set*{T_1, \hdots, T_n} \subset\Cp_2(K,\epsilon)\)
such that \(\Tcal\) \(\epsilon\)-quasi-tiles any
\(A\in\Cp_2(T_{\mathrlap{n}}{\mathstrut}^{-1} T_n, \epsilon')\).
\end{Lemma}
\begin{proof}
This is \cite[p.\ 30, Theorem 3 and Remark]{OW87},
although Ornstein and Weiss write \(\set*{F_1, \hdots, F_K}\)
instead of \(\set*{T_1, \hdots, T_n}\),
and invert the definition of \(\partial_K(A)\) as described in \ref{Ci definitions}.
\end{proof}

\begin{Lemma}\label{Lemma disjointified are still invariant}
If \(S_i \in \Cp_1(K,\epsilon)\)
and \(R_i \subset S_i\) with \(\norm*{\mu\sub{R_i} - \mu\sub{S_i}}_1 < \epsilon\),
then \(R_i \in \Cp_1(K,3\epsilon)\).
\end{Lemma}
\begin{proof}
\(\mfrac{|K R_i \setminus R_i|}{|R_i|}
	\leq \mfrac{|S_i|}{|R_i|}\cdot \mfrac{|KS_i \setminus S_i|}{|S_i|}
		+ \mfrac{|S_i \setminus R_i|}{|R_i|}
	< (1+\epsilon)\cdot\epsilon + \norm*{ \mu\sub{R_i}  - \mu\sub{S_i}}_1
	< 3\epsilon\).
\end{proof}

\begin{Theorem}\label{Theorem Unimodular Case}
If $G$ is noncompact unimodular, $\TLIM_0(G) = \TLIM(G)$.
\end{Theorem}
\begin{proof}
Pick \((\Pcal, K, \epsilon)\) and \(\mu\in\TLIM_0(G)\),
say \(\Pcal = \set*{E_1, \hdots, E_p}\).
Let \(V = \set*{x\in\Rbb^p : \|x\|_1 = 1}\).
For \(m\in\Mcal(G)\), let \(v(m) = \begin{bmatrix} m(E_1) & \hdots & m(E_p) \end{bmatrix} \in V\).
Thus \(m\in\Ncal(\mu, \Pcal,\epsilon)\) iff \(\norm*{v(m) - v(\mu)}_\infty < \epsilon\).
Let \(D\subset V\) be a finite \(\epsilon\)-dense subset, \ie\ 
for each \(v\in V\) there exists \(d\in D\) with \(\norm*{v-d}_\infty < \epsilon\).
By Lemma~\ref{Ornstein Weiss Lemma}, obtain
\(\Tcal = \set*{T_1, \hdots, T_n} \subset \Cp_2(K,\epsilon)\) and \(\epsilon' > 0\).
Let \(M = \# D\cdot|T_n| \cdot \epsilon^{-1}\).
Notice that \(M\) depends only on \((p, K, \epsilon)\).
By Lemma~\ref{Lemma A can be large},
pick \(B\in\Cp_2\paren{T_{\mathrlap{n}}{\mathstrut}^{-1} T_n,\epsilon'}\)
with \(|B| \geq M\) and \(\mu\sub{B}\in\Ncal(\mu,\Pcal,\epsilon)\).
The goal is to construct \(A\subset B\) such that \(A\in\Cp_1(K,3\epsilon)\),
\(|A|/M \in (1-\epsilon, 1]\), and
\(\norm*{v\paren*{\mu\sub{A}} - v\paren*{\mu\sub{B}}}_\infty < 5\epsilon\).
This implies \(\mu\sub{A}\in\Ncal(\mu,\Pcal,6\epsilon)\).
Since \(\epsilon\) was arbitrary, \(\TLIM_0(G) = \TLIM(G)\) will follow
by Theorem~\ref{Theorem Hindman and Strauss Technique}.

As in Definition~\ref{Definition Quasitiling},
let \(R = \bigsqcup_{i=1}^N R_i\) with \(\norm*{\mu\sub{A} - \mu\sub{R}}_1 < \epsilon\).
Clearly \(\norm*{v\paren*{\mu\sub{A}} - v\paren*{\mu\sub{R}}}_\infty < \epsilon\).
By Lemma~\ref{Lemma disjointified are still invariant}, each \(R_i\) is in \(\Cp_1(K,3\epsilon)\).
Since they are disjoint, any union of them is in \(\Cp_1(K,3\epsilon)\) as well.
The crucial detail is that each $|R_i|$ is at most $|T_n|$.
Let \(r_i = |R_i|/|R|\), so \(\sum r_i = 1\) and \(\mu\sub{R} = \sum r_i \cdot \mu\sub{R_i}\).

For each \(R_i\), pick \(d_i \in D\) with
	\(\norm*{d_i - v\paren*{\mu\sub{R_i}}}_\infty < \epsilon\).
For each \(d\in D\), let \(C_d = \set*{i : d_i = d}\) and \(c_d = \sum \set*{r_i : d_i = d}\).
Thus \(\norm*{v\paren*{\mu\sub{R}} - \sum_{d\in D} c_d \cdot d}_\infty < \epsilon\).
Let \(S_d \subset C_d\) be a maximal subset such that
\(s_d = \sum \set*{r_i : i \in S_d} \leq \mfrac{M}{|R|} c_d\).
By the maximality of \(S_d\),
\(0 \leq \big( c_d - \mfrac{|R|}{M} s_d\big) < \mfrac{\epsilon}{\# D}\).
Otherwise we could add another \(i\) from \(C_d\) to \(S_d\),
increasing \(\mfrac{|R|}{M} s_d\) by \(\mfrac{|R|}{M} r_i = \mfrac{|R_i|}{M}
	\leq \mfrac{|T_n|}{M} = \mfrac{\epsilon}{\# D}\).
Let \(A = \bigcup_{d\in D} \bigcup_{i \in S_d} R_i\),
so that \(\mu\sub{A} = \sum_{d\in D} \sum_{i\in S_d} \mfrac{|R|}{|A|} r_i \cdot \mu\sub{R_i}\).
Notice \(\mfrac{|A|}{M} = \sum_{d\in D} \sum_{i \in S_d} \mfrac{|R|}{M}\mfrac{|R_i|}{|R|}
	= \sum_{d\in D} \mfrac{|R|}{M}s_d \in (1-\epsilon, 1]\).
Finally, we apply the triangle inequality:
\(\bigl\|v\paren*{\mu\sub{B}} - v\paren*{\mu\sub{A}} \bigr\|_\infty
	\leq \bigl\|v\paren*{\mu\sub{B}} - v\paren*{\mu\sub{R}} \bigr\|_\infty
	+\ \bigl\|v\paren*{\mu\sub{R}} - \sum_{d\in D} c_d \cdot d \bigr\|_\infty
	+\ \norm*{\sum_{d\in D} \paren*{c_d - \mfrac{|R|}{M} s_d} \cdot d}_\infty
	+\ \norm*{\sum_{d\in D} \paren*{\mfrac{|R|}{M} - \mfrac{|R|}{|A|}} s_d \cdot d}_\infty
	+\ \norm*{ \sum_{d\in D}\sum_{i\in S_d} \mfrac{|R|}{|A|} r_i \cdot
		\paren*{d - v\paren*{\mu\sub{R_i}}} }_\infty
	< 5\epsilon.
\)
\end{proof}


\section{Easy cases}
\begin{Theorem}\label{Theorem when G is large}
If $\kappa(G)>\Nbb$, $\TLIM_0(G) = \TLIM(G)$.
\end{Theorem}
\begin{proof}
Pick \(\mu,\nu\in\TLIM_0(G)\) and \((\Pcal, K, \epsilon)\).
By the same technique as in Lemma~\ref{Lemma many disjoint sets}, obtain mutually disjoint sets
\(\set*{A_\alpha, B_\alpha : \alpha < \kappa(G)} \subset \Cp(K,\epsilon)\),
such that \(\mu\sub{A_\alpha}\in \Ncal(\mu,\Pcal, \epsilon)\),
and \(\mu\sub{B_\alpha} \in \Ncal(\nu,\Pcal,\epsilon)\) for each \(\alpha\).

Let $r = (1+\epsilon)^{1/2}$, and let $I_n = [r^n, r^{n+1})$ for $n\in\Zbb$.
By the pigeonhole principle, there exist \(m,n\in\Zbb\) such that
\(\set*{\alpha : |A_\alpha| \in I_m}\) and \(\set*{\alpha : |B_\alpha| \in I_m}\) are infinite.
Pick \(M,N\in\Nbb\) so that \(M/N \in I_{m-n}\).
Pick \(A_1, \hdots, A_N\) from \(\set*{A_\alpha : |A_\alpha|\in I_m}\)
and \(B_1, \hdots, B_M\) from \(\{B_\alpha : |B_\alpha| \in I_n\}\).
Let \(A = A_1 \cup \hdots \cup A_N\) and \(B = B_1 \cup \hdots \cup B_M\).
Thus $|A| / |B| \in (1-\epsilon, 1+\epsilon)$.
Since the $A_i$'s and $B_i$'s are mutually disjoint, we have $A,B, A\cup B\in\Cp(K,\epsilon)$.
Clearly \(\mu\sub{A}\in\Ncal(\mu,\Pcal,\epsilon)\) and
\(\mu\sub{B}\in\Ncal(\nu,\Pcal, \epsilon)\).
By Lemma~\ref{Lemma Mean Approximation} and the triangle inequality,
\(\mu\sub{A\cup B} \in \Ncal\bigl(\frac12(\mu + \nu), \Pcal, 4\epsilon\bigr)\).
Since \(\epsilon\) was arbitrary, we conclude \(\frac12(\mu + \nu) \in \TLIM_0(G)\).
\end{proof}

\begin{Theorem}\label{Theorem Non-Unimodular}
If $G$ is $\sigma$-compact non-unimodular, $\TLIM_0(G) \neq \TLIM(G)$.
\end{Theorem}
\begin{proof}
Recall that the modular function \(\Delta\) is defined by \(|Ex| = |E|\Delta(x)\).

Let \(K_1 \subset K_2 \subset \hdots\) be a sequence of compacta with \(\bigcup_n K_n = G\).
For each $n$, pick \(F_n \in \Cp\big(K_n, \frac1n\big)\),
then pick $x_n, y_n\in G$ such that $F_n x_n \subset \{x : \Delta(x) \geq 1\}$
and \(F_n y_n \subset \set*{x : \Delta(x) \leq \min\paren*{1, |F_n|^{-1} 2^{-n}}}\).
Let $X = \bigcup_n F_n x_n$ and $Y = \bigcup_n F_n y_n$.
By construction, \(X\cap Y = \varnothing\) and  \(|Y| \leq 1\).
Let $\mu$ be an accumulation point of $\seq*{\mu\sub{F_n x_n}}$
and $\nu$ an accumulation point of $\seq*{\mu\sub{F_n y_n}}$.
Thus $\mu,\nu\in \TLIM_0(G)$.
Let $m = \frac12(\mu + \nu)$.
Notice $\mu(X) = \nu(Y) = 1$ and $\mu(Y) = \nu(X) = 0$, hence $m(X) = m(Y) = \mfrac12$.

Suppose \(m\in \TLIM_0(G)\).
Let $K$ be any compact set with $|K| \geq \frac72$.
Define \(\epsilon = \frac16\) and \(\Pcal = \{X, Y\}\).
Pick $A\in\Cp_1(K, \epsilon)$ with $\mu\sub{A} \in \Ncal(m, \Pcal, \epsilon)$.
Now \(\mfrac{1}{|A|}
	\geq \mfrac{|Y|}{|A|}
	\geq \mfrac{|Y\cap A|}{|A|}
	= \mu\sub{A}(Y) > m(Y) - \epsilon = \frac13\), so \(|A| < 3\).
On the other hand, $\mu\sub{A}(X) > \frac13$,
so $|A \cap X| \neq 0$. Say $a \in A \cap X$.
Now $\frac72 \leq |K| \leq |Ka| \leq |KA| \leq |A| + |KA \setminus A| < (1+\epsilon)|A|$,
	so $3 < |A|$, a contradiction.
Hence $m\in\TLIM(G) \setminus \TLIM_0(G)$.
\end{proof}
\section{Non-topological invariant means}
In this section, triples of the form \((\Pcal, K, \epsilon)\) range over
	finite measurable partitions \(\Pcal\) of \(G\),
	finite sets \(K\subset G\),
	and \(\epsilon \in (0,1)\).
In these terms, we define
\begin{center}
\(\LIM_0(G) = \set*{
	\mu\in\LIM(G) : \big(\forall \paren*{\Pcal, K, \epsilon}\big)\ 
	\big(\exists A\in\Cp(K,\epsilon)\big)\ 
	\mu\sub{A} \in \Ncal(\mu, \Pcal, \epsilon)}\).
\end{center}

\begin{Lemma}[{\cite[Theorem 2.12]{HindmanDensity}}]\label{Lemma LIM0 is total in LIM}
The closed convex hull of $\LIM_0(G)$ is all of $\LIM(G)$.
\end{Lemma}

\begin{Theorem}\label{Theorem LIM when G is large}
If $\kappa(G) > \Nbb$, then $\LIM_0(G) = \LIM(G)$.
\end{Theorem}
\begin{proof}
The proof is identical to that of Theorem~\ref{Theorem when G is large},
	but with Lemma~\ref{Lemma LIM0 is total in LIM}
	in place of Lemma~\ref{Lemma TLIM0 is total in TLIM}.
\end{proof}

\begin{point}\label{Refinement}
Fix a finite measurable partition \(\Pcal = \set*{E_1, \hdots, E_p}\),
	and a finite set \(X = \{x_1, \hdots, x_n\}\), with \(x_1 = e\).
For \(C = (c_1, \hdots, c_n) \in \{1, \hdots, p\}^n\), define
	\(E(C) = \set*{y \in G : x_1 y \in E_{c_1}, \hdots, x_n y \in E_{c_n}}
		= \bigcap_{k=1}^n x_k^{-1} E_{c_k}\).
Thus \(\Qcal = \set*{E(C) : C\in \{1,\hdots,p\}^n}\) is a refinement of \(\Pcal\).
Notice \(E_i = \bigcup\{E(C) : c_1 = i\}\)
and \(x_k^{-1} E_i = \bigcup\{E(C) : c_k = i\}\).

The idea of refining $\Pcal$ this way is due to \cite{WeakNotStrong}.
\end{point}

\begin{Theorem}\label{Theorem LIM0}
If $G$ is non-discrete but amenable-as-discrete, then $\LIM_0(G) = \LIM(G)$.
\end{Theorem}
\begin{proof}
Pick \((\Pcal, K, \epsilon)\) and \(\mu\in\LIM(G)\).
Let $X = \{x_1, \hdots, x_n\}$ be $(K, \epsilon)$-invariant, with $x_1 = e$.
Let $\Qcal$ refine $\Pcal$ as in \ref{Refinement},
and let \(\set{F_1, \hdots, F_q} = \set{F\in \Qcal : |F| > 0}\).
Let $m=\min\{1, |F_1|, \hdots, |F_q|\}$ and $c = m / (n^2 q)$.
By Lemma~\ref{Lemma Very Small S}, pick $S_1 \subset F_1$
such that $0 <|S_1| \leq c$ and $S_1 S_1^{-1} \cap X^{-1} X = \{e\}$.
For \(k < q\), inductively choose
\(S_{k+1} \subset F_{k+1} \setminus X^{-1} X (S_1 \cup \hdots \cup S_k)\)
with \(0<|S_{k+1}| \leq c\) and \(S_{k+1} S_{k+1}^{-1} \cap X^{-1} X = \{e\}\).
This is possible, since
\(|F_{k+1} \setminus X^{-1} X (S_1 \cup \hdots \cup S_k)| \geq m - n^2kc > 0\).

Now let \(m' = \min\set*{|S_1|, |S_2|, \hdots, |S_q|}\).
For each $i$, choose $S_i' \subset S_i$ with \(|S_i'| = m'\cdot \mu(E_i)\),
then let $S = S_1' \cup \hdots \cup S_q'$.
By construction, $XS = x_1 S \sqcup \hdots \sqcup x_n S$.
For $y\in K$, $|yXS \Sdiff XS| / |XS| = \#(yX \Sdiff X) / \# (X) < \epsilon$,
	hence \(XS \in \Cp(K, \epsilon)\).
Notice $\mu\sub{XS} = \frac{1}{n} \sum_{k=1}^n \mu\sub{x_k S}$.
For each $i$ and $k$, $\mu\sub{x_k S}(E_i)
	= \mu\sub{S}(x_k^{-1} E_i)
	= \sum \set*{\mu\sub{S}(E(C)) : c_k = i}
	= \sum \set*{\mu(E(C)) : c_k = i}
	= \mu(x_k^{-1} E_i) = \mu (E_i)$.
Hence $\mu\sub{XS}(E_i) = \mu(E_i)$.
\end{proof}
\bibliographystyle{alpha}
\bibliography{TLIM=TLIM0}
\end{document}